\documentclass[11pt]{article}
\usepackage{amsmath,amssymb, amscd}

\newtheorem{propo}{{\bf Proposition}}[section]
\newtheorem{coro}[propo]{{\bf Corollary}}
\newtheorem{lemma}[propo]{{\bf Lemma}} \newtheorem{theor}[propo]{{\bf
Theorem}} \newtheorem{ex}{{\sc Example}}[section]

\newenvironment{proof}{{\bf Proof.}}{$\Box$}

\def\C{{\mathbb C}}
\def\N{{\mathbb N}}
\begin{document}

\vspace*{1.0in}

\begin{center} LIE ALGEBRAS WITH NILPOTENT LENGTH GREATER THAN THAT OF EACH OF THEIR SUBALGEBRAS
\end{center}
\bigskip

\begin{center} DAVID A. TOWERS 
\end{center}
\bigskip
\centerline {Department of
Mathematics, Lancaster University} \centerline {Lancaster LA1 4YF,
England} \centerline {d.towers@lancaster.ac.uk}
\bigskip

\begin{abstract}
The main purpose of this paper is to study the finite-dimensional solvable Lie algebras described in its title, which we call {\em minimal non-${\mathcal N}$}. To facilitate this we investigate solvable Lie algebras of nilpotent length $k$, and of nilpotent length $\leq k$, and {\em extreme} Lie algebras, which have the property that their nilpotent length is equal to the number of conjugacy classes of maximal subalgebras. We characterise the minimal non-${\mathcal N}$ Lie algebras in which every nilpotent subalgebra is abelian, and those of solvability index $\leq 3$.
\par

\end{abstract}

\section{Introduction}
Let $L$ be a Lie algebra, and let $L^{(0)}=L$, $L^{(i+1)}=[L^{(i)},L^{(i)}]$ be its derived series. Recall that $L$ is {\em solvable} if there exists $r$ such that $L^{(r)}=0$; the smallest such $r$ is called the {\em derived length} of $L$. Similarly, $L^1=L$, $L^{i+1}=[L^i,L]$ is the lower central series of $L$; $L$ is {\em nilpotent} of {\em nilpotency index $r$} if $L^{r+1}=0$ but $L^r \neq 0$. Throughout, $L$ will denote a finite-dimensional solvable Lie algebra over a field $F$.  The symbol `$\oplus$' will denote an algebra direct sum, whilst `$\dot{+}$' will denote a direct sum of the underlying vector space structure alone. If $U$ is a subalgebra of $L$ we define $U_L$, the {\em core} (with respect to $L$) of $U$, to be the largest ideal of $L$ contained in $U$. We say that $U$ is {\em core-free} in $L$ if $U_L = 0$. 
\par

We denote the nilradical of $L$ by $N(L)$.
We define the {\em upper nilpotent series} of $L$ by 
\[ N_0(L)=0, \hspace{.5cm} N_i(L)/N_{i-1}(L)=N(L/N_{i-1}(L)) \hbox{ for } i=1,2, \ldots
\]
The {\em nilpotent length}, $n(L)$, of $L$ is the smallest integer $n$ such that $N_n(L)=L$. 
\par

We define the {\em nilpotent residual}, $\gamma_{\infty}(L)$, of $L$ be the smallest ideal of $L$ such that $L/\gamma_{\infty}(L)$ is nilpotent. Clearly this is the intersection of the terms of the lower central series for $L$. Then the {\em lower nilpotent series} for $L$ is the sequence of ideals $\Gamma_i(L)$ of $L$ defined by $\Gamma_0(L) = L$, $\Gamma_{i+1}(L) = \gamma_{\infty}(\Gamma_i(L))$ for $i \geq 0$. First we note that this series has the same length as that of the upper nilpotent series. 

\begin{lemma}\label{l:length} Suppose that $r$ is the smallest integer such that $N_r(L)=L$, and that $s$ is the smallest integer such that $\Gamma_s(L)=0$. Then 
\begin{itemize}
\item[(i)] $\Gamma_{s-i}(L)\subseteq N_i(L)$ for $i=0,\ldots,r$;
\item[(ii)] $\Gamma_i(L)\subseteq N_{r-i}(L)$ for $i=0,\ldots,s$; and
\item[(iii)] $r=s$.
\end{itemize}
\end{lemma}
\begin{proof}\begin{itemize}
\item[(i)] Clearly $\Gamma_{s-1}(L)$ is a nilpotent ideal of $L$ and so $\Gamma_{s-1}(L)\subseteq N_1(L)$. Suppose that $\Gamma_{s-k}(L)\subseteq N_k(L)$ for some $k\geq 1$. then
\[ \frac{\Gamma_{s-k-1}(L)+N_k(L)}{N_k(L)}
\] is a nilpotent ideal of $L/N_k(L)$ and so is contained in $N_{k+1}(L)/N_k(L)$. Hence $\Gamma_{s-k-1}(L)\subseteq N_{k+1}(L)$, and the result follows.
\item[(ii)] Clearly $\Gamma_1(L)\subseteq N_{r-1}(L)$. Suppose that $\Gamma_k(L)\subseteq N_{r-k}(L)$ for some $k\geq 1$. Then
\[ \frac{\Gamma_k(L)+N_{r-k-1}(L)}{N_{r-k-1}(L)}\cong \frac{\Gamma_k(L)}{\Gamma_k(L)\cap N_{r-k-1}(L)}
\] is nilpotent, whence $\Gamma_{k+1}(L)\subseteq N_{r-k-1}(L)$, and the result follows.
\item[(iii)] From (i) we have that $\Gamma_{s-r}(L)\subseteq N_r(L)$, so $s-r\geq 0$; from (ii) we see that $\Gamma_s(L)\subseteq N_{r-s}(L)$, so $r-s\geq 0$. Thus $r=s$.
\end{itemize}
\end{proof}
\bigskip

Although the upper and lower nilpotent series have equal lengths, $n$ say, we do not necessarily have that $\Gamma_{n-i}(L)=N_i(L)$ for $i=0,\ldots,n$, as the following example shows.

\begin{ex} Let $L$ be the metabelian Lie algebra over $\C$ with basis $x_1$, $x_2$, $x_3$, $x_4$ and non-zero products $[x_1,x_2]=x_3$, $[x_1,x_3]=x_3$, $[x_1,x_4]=x_4$, $[x_2,x_3]=x_4$. Then $N_1(L)=\C x_2+\C x_3+\C x_4$, $N_2(L)=L$, $\Gamma_1(L)=\C x_3+\C x_4$, $\Gamma_2(L)=0$.
\end{ex}

From now on we choose to work with the upper nilpotent series. In section $2$ we investigate properties of this series, particularly its relationship to maximal subalgebras, and of Lie algebras with nilpotent length $k$ or $\leq k$. In considering factor algebras, a complication arises because, unlike the situation in group theory, there are solvable Lie algebras $L$ in which, for an ideal $I$ of $L$, $N(I)$ may not be contained in $N(L)$. To overcome this obstacle we introduce the notions of nilregular and strongly nilregular subalgebras of $L$; in particular, it is shown that if a maximal subalgebra of $L$ has a strongly nilregular core, its nilpotent length is at most one less than that of $L$. The section concludes with a fundamental decomposition theorem for Lie algebras with a given nilpotent length.
\par

In section $3$ we introduce the class of extreme Lie algebras in which $N_i(L)/\phi_i(L)$ is a chief factor for each $i=1, \ldots, n(L)$. These are characteriseded in relation to the decomposition result from the previous section and described explicitly in two special cases. 
\par

The final section then considers the algebras in the title of the paper. By considering their relationship to extreme Lie algebras the minimal non-${\mathcal N}$ Lie algebras in which every nilpotent subalgebra is abelian, and those of solvability index $\leq 3$, are characterised. The last result is that a homomorphic image of a minimal non-${\mathcal N}$ Lie algebra is minimal non-${\mathcal N}$ if it has a complemented minimal ideal.

Much of this is inspired by corresponding work in group theory in \cite{frick} and \cite{c-f-h}, but there are significant differences encountered in the Lie case.

\section{Properties of the upper nilpotent series}
The {\em Frattini subalgebra} of $L$, $\phi(L)$, is the intersection of the maximal subalgebras of $L$. Since $L$ is solvable, this is an ideal of $L$ (\cite[Lemma 3.4]{b-g-h}). 
The {\em Frattini series} of $L$ is given by
\[ \phi_i(L)/N_{i-1}(L)= \phi(L/N_{i-1}(L)) \hbox{ for } i=1,2, \ldots
\]
\begin{lemma}\label{l:nat} Let $B$ be an ideal of $L$ with $B \subseteq \phi(L)$. Then $N_i(L/B)=N_i(L)/B$ and $\phi_i(L/B)=\phi_i(L)/B$ for every $i=1,2, \ldots, n(L)$.
\end{lemma}
\begin{proof} We have $N(L)/B=N(L/B)$, by \cite[Theorem 5]{barnes}, and $\phi(L)/B=\phi(L/B)$, by \cite[Proposition 4.3]{frat}. Suppose that $N_k(L)/B=N_k(L/B)$ and $\phi_k(L)/B= \phi_k(L/B)$. Then $B \subseteq \phi_{k+1}(L)$ and
\[ \frac{N_{k+1}(L/B)}{N_k(L/B)}=N\left(\frac{L/B}{N_k(L)/B}\right)=\frac{N_{k+1}(L)/B}{N_k(L)/B},
\]
whence $N_{k+1}(L)/B=N_{k+1}(L/B)$.
Similarly,
\[ \frac{\phi_{k+1}(L/B)}{N_k(L/B)}=\phi\left(\frac{L/B}{N_k(L)/B}\right)=\frac{\phi_{k+1}(L)/B}{N_k(L)/B},
\]
which yields that $\phi_{k+1}(L)/B=\phi_{k+1}(L/B)$.
\end{proof}

\begin{lemma}\label{l:quot} If $A$ is an ideal of $L$ with $N_{r-1}(L)\subseteq A \subseteq N_r(L)$, then $n(L/A)=n(L)-r$ or $n(L)-r+1$.
\end{lemma}
\begin{proof} Put $K_i/A=N_i(L/A)$. Then it is easy to see that $K_i/K_{i-1}=N(L/K_{i-1})$, and a straightforward induction argument shows that 
\[ N_{r-1}(L)\subseteq A\subseteq N_r(L)\subseteq K_1\subseteq N_{r+1}(L)\subseteq \ldots \subseteq K_{n(L)-r}\subseteq N_{n(L)}(L).
\] If $K_{n(L)-r}=N_{n(L)}(L)$ we have $n(L/A)=n(L)-r$; otherwise, $n(L/A)=n(L)-r+1$.
\end{proof}

\begin{lemma}\label{l:max} Let $M$ be a maximal subalgebra of the solvable Lie algebra $L$. Then 
\begin{itemize}
\item[(i)] $N_i(L) \cap M \subseteq N_i(M)$;
\item[(ii)] $N_i(M)_L \subseteq N_i(L)$;
\item[(iii)] if $N_i(L) \subseteq M$ then $N_i(M)_L=N_i(L)$; 
\item[(iv)] if $k$ is the smallest positive integer such that $N_k(L) \not \subseteq M$ then $N_k(M)_L=N_k(L) \cap M$; and
\item[(v)] if $N(L) \subseteq M$ then $N(M)$ acts nilpotently on $L$.
\end{itemize}
\end{lemma}
\begin{proof} \begin{itemize} \item[(i)] This is a straightforward induction proof.
\item[(ii)] It is easy to see that this holds for $i=1$. So suppose it holds for $i<k$ ($k \geq 2$). Then there is an $r \in \N$ such that $(N_k(M)_L)^r \subseteq N_k(M)^r \subseteq N_{k-1}(M)$, so $(N_k(M)_L)^r \subseteq N_{k-1}(M)_L \subseteq N_{k-1}(L)$ by the inductive hypothesis. Thus $N_k(M)_L \subseteq N_k(L)$ and the result follows by induction.
\item[(iii)] This follows from (i) and (ii).
\item[(iv)] Suppose that $k$ is the smallest positive integer such that $N_k(L) \not \subseteq M$. Then $L=N_k(L)+M$ and $N_{k-1}(L) \subseteq M$, so $\phi_k(L) \subseteq M$. Moreover, $N_k^2 \subseteq \phi_k(L) \subseteq M$, by Lemma \ref{l:nat} and \cite{frat}. It follows that $N_k(L) \cap M$ is an ideal of $L$ and hence that $N_k(L) \cap M \subseteq N_k(M)_L$. The reverse inclusion follows as in (ii).
\item[(v)]  Let $N(L) \subseteq M$. We show that $N(M)$ acts nilpotently on $L$. Suppose not, and let $L=L_0 \dot{+} L_1$ be the Fitting decomposition of $L$ relative to $N(M)$. Then $L_0=M$ and $[N(L),L_1] \subseteq M \cap L_1=0$. Hence $L_1 \subseteq C_L(N(L)) \subseteq N(L)$, giving $L=M$, a contradiction.
\end{itemize}
\end{proof}
\bigskip

In general, over a field of characteristic $p>0$, $N(L)$ is not a characteristic ideal of $L$. The best known example is due to Jacobson and first appeared in \cite{sel}; however, it is not solvable. We shall see next that $N(L)$ is characteristic in $L$ whenever $\phi(L)$ is. First we need some lemmas.

\begin{lemma}\label{l:ab} Let $L$ be a Lie algebra over a field of characteristic $p>2$, let $I$ be an abelian ideal of $L$ and let $D$ be a derivation of $L$. Then $I+D(I)\subseteq N(L)$.
\end{lemma}
\begin{proof} This follows easily from \cite[Theorem 1]{mak}.
\end{proof}

\begin{lemma}\label{l:charfac} Let $I$ be a characteristic ideal of $L$, and let $D$ be a derivation of $L$. Then $\bar{D}:L/I\rightarrow L/I:x+I \mapsto D(x)+I$ is a derivation of $L/I$.
\end{lemma}
\begin{proof} This is easy to check.
\end{proof}

\begin{propo}\label{p:phifree} Let $L$ be a $\phi$-free Lie algebra over a field of characteristic $p>2$. Then $N(L)$ is a characteristic ideal of $L$.
\end{propo}
\begin{proof} Since $L$ is $\phi$-free, $N(L)=A_1\oplus \ldots \oplus A_r$, where $A_1, \ldots, A_r$ are the minimal abelian ideals of $L$, by \cite[Theorem 7.4]{frat}. But $D(A_i)\subseteq N(L)$ for all $D\in$ Der$(L)$ and $i=1, \ldots, r$, whence the result.
\end{proof}

\begin{coro}\label{c:phifree}  Let $L$ be a Lie algebra over a field of characteristic $p>2$, and suppose that $\phi(L)$ is characteristic in $L$. Then $N(L)$ is characteristic in $L$.
\end{coro}
\begin{proof} This follows easily from Lemma \ref{l:charfac} and Proposition \ref{p:phifree}.
\end{proof}
\bigskip

However, there are solvable Lie algebras $L$ in which $N(L)$ is not a characteristic ideal, as the following example shows.

\begin{ex}\label{e:nonchar} Let $L$ be the four-dimensional Lie algebra with basis $x_1$, $x_2$, $x_3$, $x_4$ and non-zero products $[x_4,x_2]=x_1$, $[x_3,x_1]=x_1$ and $[x_3,x_2]=x_2$ over a field $F$ of characteristic $2$. Then $N(L)=Fx_1+Fx_2+Fx_4$ and if $D(x_1)=x_2$, $D(x_2)=0$, $D(x_3)=0$, $D(x_4)=x_3$ then the extension of $D$ to $L$ by linearity is a derivation of $L$. Clearly, $D(N(L))=Fx_2+Fx_3$. Note that $\phi(L)=Fx_1$, so this leaves open the question of whether Proposition \ref{p:phifree} holds over a field of characteristic $2$.
\par

If we form the split extension $X=Fd\dot{+} L$, where $[d,x]=D(x)$ for all $x\in L$, then $L$ is an ideal of $X$, but $N(L)$ is not. 
\end{ex}

Consequently, we shall need the following result from \cite{nmax}.

\begin{propo}\label{p:nil} Let $I$ be a nilpotent subideal of a Lie algebra $L$ over a field $F$. If $F$ has characteristic zero, or has characteristic $p$ and $L$ has no subideal with nilpotency class greater than or equal to $p-1$, then $I \subseteq N$, where $N$ is the nilradical of $L$. 
\end{propo}

Let $L$ be a Lie algebra over a field $F$ and let $U$ be a subalgebra of $L$. We call the largest integer $r$ such that $N_{r}(L)\subseteq U$ the {\em compatibility index} of $U$. As in \cite{gennil}, if $F$ has characteristic $p>0$, we will call $U$ {\em nilregular} if the nilradical of $U$  has nilpotency class less than $p-1$. If $U$ has compatibility index $r$, we say that $U$ is {\em strongly nilregular} if $N_k(U)/N_{k-1}(U)$ has nilpotency class less than $p-1$ for $k=1, \ldots, r$. If $F$ has characteristic zero we regard every subalgebra of $L$ as being nilregular. Then we have the following result.

\begin{propo}\label{p:nilregular} If $I$ is a nilregular ideal of $L$ then $N(I)\subseteq N(L)$.
\end{propo}
\begin{proof} This is \cite[Proposition 2.1]{gennil}.
\end{proof}
\bigskip

We shall call $L$ {\em primitive} if it has a core-free maximal subalgebra. It is said to be {\em primitive of type $1$} if it has a unique minimal ideal that is abelian; since $L$ is solvable this is the only type that can occur here (see \cite{chief}). If $B$ is an ideal of $L$ and $U/B$ is a subalgebra of $L/B$, the {\em centraliser} of $U/B$ in $L$ is $C_L(U/B)=\{x\in L : [x,U]\subseteq B\}$.

\begin{propo}\label{p:nmax} Let $L$ be a solvable Lie algebra over a field $F$, and let $M$ be a maximal subalgebra of $L$ of compatibility index $r$. If $M_L$ is strongly nilregular, then $N_i(M)=N_i(M_L)=N_i(L)$ for $i=1,\ldots , r$.
\end{propo}
\begin{proof} First we show that if $M_L$ is nilregular then $N(M)=N(M_L)=N(L)$. We have that $N(M)$ acts nilpotently on $L$, by Lemma \ref{l:max} (v). Now $L/M_L$ is primitive of type $1$, and so $L/M_L=A/M_L \dot{+} M/M_L$, where $A/M_L$ is the unique minimal ideal of $L/M_L$ and is self-centralising, by \cite[Theorem 1.1]{chief}. Since $L=A+M$, we have that $A+N(M)$ is an ideal of $L$, and so $[A,N(M)]+M_L=[A,A+N(M)]+M_L=M_L$ or $A$, since $A/M_L$ is a chief factor of $L$. 
\par

The former implies that $[A,N(M)] \subseteq M_L$, whence 
\[ \frac{N(M)+M_L}{M_L} \subseteq C_{L/M_L}\left( \frac{A}{M_L} \right) = \frac{A}{M_L}.
\] Thus $N(M) \subseteq A \cap M \subseteq M_L$. 
\par

The latter gives that $[A,N(M)]+M_L=A$. But then an easy induction shows that $A \subseteq A$ (ad\,$N(M))^r +M_L$ for every $r \in \N$, whence $A=M_L$ since $N(M)$ acts nilpotently on $L$. But this is impossible, so $N(M) \subseteq M_L$. 
\par

It follows that $N(M)\subseteq N(M_L)$ and hence $N(M) \subseteq N(L)$, by Proposition \ref{p:nilregular}. Hence  $N(M)=N(M_L)=N(L)$.
\par

So suppose now that $N_k(M)=N_k(M_L)=N_k(L)$ for some $1\leq k<r$. Then $(M/N_k(L))_L=M_L/N_k(M_L)$ is nilregular, so, by the above,
\[ N\left(\frac{M}{N_k(M)}\right)=N\left(\frac{M}{N_k(L)}\right) =N\left(\frac{M_L}{N_k(M_L)}\right)=N\left(\frac{L}{N_k(L)}\right),
\] whence
\[ \frac{N_{k+1}(M)}{N_k(M)}=\frac{N_{k+1}(M_L)}{N_k(M_L)}=\frac{N_{k+1}(L)}{N_k(L)}
\] and $N_{k+1}(M)=N_{k+1}(M_L)=N_{k+1}(L)$. The result follows by induction.
\end{proof}
\bigskip

Note that the above result is not true for all maximal subalgebras, as is shown in the next example.

\begin{ex}\label{e:nonchar2} Let $X$ be as in Example \ref{e:nonchar}. Then $M=Fd+Fx_1+Fx_2$ and $L$ are both maximal subalgebras of $X$ of compatibility index $1$. However, $N(L) \neq N(X)=Fx_1+Fx_2$, and $M_X=Fx_1+Fx_2$, so $N(M_L)=M_L\neq M=N(M)$.
\end{ex}

Let ${\mathcal N(k)}$, ${\mathcal N({\leq k)}}$ denote the classes of Lie algebras of nilpotent length $k$ and of nilpotent length $\leq k$ respectively. Of course, over a field of characteristic zero, every Lie algebra $L \in {\mathcal N({\leq 2})}$. However, over a field of characteristic $p>0$ it is easy to construct Lie algebras $L \in {\mathcal N(k)}$ for any $k \in \N$.
\par

A class ${\mathcal H}$ of finite-dimensional solvable Lie algebras is called a {\em homomorph} if ${\mathcal H}$ contains, along with an algebra $L$, all epimorphic
images of $L$. A homomorph ${\mathcal H}$ is called a {\em formation} if $L/A$, $L/B\in {\mathcal H}$ implies that $L/A\cap B \in {\mathcal H}$, where $A$, $B$ are ideals of  $L$. A formation ${\mathcal H}$ is said to be {\em saturated} if $L/\phi(L) \in {\mathcal H}$ implies that $L \in {\mathcal H}$. 

\begin{propo}\label{p:satfor} The class ${\mathcal N({\leq k})}$ is saturated formation for each $k \geq 1$. 
\end{propo}
\begin{proof} It is shown that ${\mathcal N(1)}$ is a saturated formation in \cite[Lemma 3.7]{b-g-h}. Suppose that it holds for $k=r$. Then ${\mathcal N}(\leq r+1)$ is clearly a homomorph. Suppose that $L/A$, $L/B \in {\mathcal N}(\leq r+1)$. Let $S/A=N(L/A)$ and $T/B=N(L/B)$. Then $L/S$, $L/T \in {\mathcal N}(\leq r)$, so $L/S\cap T \in {\mathcal N}(\leq r)$. But there exist $m$, $n\in \N$ such that $S^m\subseteq A$ and $T^n\subseteq B$, so $(S\cap T)^{m+n}\subseteq A\cap B$. Hence $L/A\cap B \in {\mathcal N}(\leq r+1)$, and ${\mathcal N}(\leq r+1)$ is a formation.
\par

Suppose now that $L/\phi(L) \in {\mathcal N}(\leq r+1)$. Then $N(L/\phi(L))=N(L)/\phi(L)$, by \cite[Theorem 6.1]{frat}, so $L/N(L) \in {\mathcal N}(\leq r)$ and $L \in {\mathcal N}(\leq r+1)$. It follows that ${\mathcal N}(\leq r+1)$ is saturated.
\end{proof}

\begin{coro}\label{c:satfor} Let $L \in {\mathcal N(k)}$ have more than one minimal ideal. Then there is at least one minimal ideal $A$ of $L$ such that $L/A \in {\mathcal N(k)}$.
\end{coro}
\begin{proof} If $A_1, \ldots, A_n$ are minimal ideals of $L$, where $n>1$, and $L/A_i \in {\mathcal N({\leq k-1})}$ for all $1 \leq k \leq n$, then $L \in {\mathcal N({\leq k-1})}$, by Proposition \ref{p:satfor}.
\end{proof} 

\begin{propo}\label{p:max}  Let $L$ be a solvable Lie algebra over a field $F$. If $M$ is a maximal subalgebra of $L$ for which $M_L$ is strongly nilregular,  then $n(M)=n(L)-i$ where $i \in \{0,1\}$.
\end{propo}
\begin{proof} Let $r$ be the compatibility index of $M$. Then $L=M+N_{r+1}(L)$ and $L/N_{r+1}(L)\cong M/M\cap N_{r+1}(L)=M/N_{r+1}(M)_L$, by Lemma \ref{l:max} (iv). Now $N_r(L)=N_r(M)$ by Proposition \ref{p:max}, so $N_r(M)\subseteq N_{r+1}(M)_L\subseteq N_{r+1}(M)$. It follows from Lemma \ref{l:quot} that $n(M/N_{r+1}(M)_L)=n(M)-r-1$ or $n(M)-r$. Hence $n(L)-r-1=n(M)-r-1$ or $n(M)-r$, which gives the result.
\end{proof}

\begin{lemma}\label{l:decomp} Let $L \in {\mathcal N}(n)$. Then we can write $L=N_k(L)+U_k$, where $U_k$ is a subalgebra of $L$, $N_k(L)\cap U_k \subseteq \phi(U_k)=\phi_{k+1}(L)\cap U_k$, $U_k\subseteq U_{k-1}$ for each $k=1,\ldots,n$. Moreover, $N(U_k)=N_{k+1}(L)\cap U_k$  for each $k=0,\ldots,n-1$.
\end{lemma}
\begin{proof} Put $U_0=L$. Then $L=N_1(L)+ U_1$ for some subalgebra $U_1$ of $L$ with $N(L)\cap U_1\subseteq \phi(U_1)$, by \cite[Lemma 4.1]{frat}. Having constructed $U_j$ we construct $U_{j+1}$ such that $U_j=N_{j+1}(L)\cap U_j+U_{j+1}$ and $N_{j+1}(L)\cap U_{j+1}\subseteq \phi(U_{j+1})$, which we can do, by using \cite[Lemma 4.1]{frat} again. Now, it is easy to see inductively that $L=N_k(L)+U_k$, $N_k(L)\cap U_k\subseteq \phi(U_k)$ and $U_k\subseteq U_{k-1}$  for each $k=1,\ldots,n$. Furthermore
\begin{align}
\frac{\phi_{k+1}(L)\cap U_k}{N_k(L)\cap U_k} & \cong \frac{N_k(L)+\phi_{k+1}(L)\cap U_k}{N_k(L)}= \frac{\phi_{k+1}(L)}{N_k(L)} \nonumber \\
  & = \phi \left(\frac{L}{N_k(L)}\right) \cong \phi\left(\frac{U_k}{N_k(L)\cap U_k}\right)= \frac{\phi(U_k)}{N_k(L)\cap U_k}, \nonumber
\end{align}
so $\phi(U_k)=\phi_{k+1}(L)\cap U_k$.
\par

Clearly $N_k(L)+N(U_k)\subseteq N_{k+1}(L)$ so $N(U_k)\subseteq N_{k+1}(L)\cap U_k$. But also, there is a natural number $r$ such that 
\[ (N_{k+1}(L)\cap U_k)^r \subseteq N_k(L)\cap U_k \subseteq \phi(U_k),
\] so $N_{k+1}(L)\cap U_k/\phi(U_k)$ is nilpotent. It follows from \cite[Theorem 6.1]{frat} that $N_{k+1}(L)\cap U_k$ is a nilpotent ideal of $U_k$, so $N_{k+1}(L)\cap U_k \subseteq N(U_k)$ and equality results.
\end{proof}
\bigskip

Next we have a fundamental decomposition result.

\begin{theor}\label{t:decomp} Let $L \in {\mathcal N}(n)$ if and only if there are nilpotent subalgebras $B_i$ of $L$ for $i=1,\ldots,n$ such that
\begin{itemize}
\item[(i)] $N_i(L)=B_1+\ldots +B_i$ for $i=1,\ldots, n$, 
\item[(ii)] $L=B_1+\ldots +B_n$, 
\item[(iii)] $[B_i,B_j]\subseteq B_i$ for $1\leq i\leq j\leq n$, and 
\item[(iv)]$N_i(L)\cap U_i\subseteq \phi(U_i)=\phi_{i+1}(L)\cap U_i$, where $U_i= B_{i+1}+\ldots +B_n$,  for $i=1,\ldots,n-1$. 
\end{itemize}
\end{theor}
\begin{proof} Let $L \in {\mathcal N}(n)$, $U_i$ be as in Lemma \ref{l:decomp} and put $B_i=N(U_{i-1})$, where $U_0=L$.
\begin{itemize}
\item[(i)] Clearly $B_1=N_1(L)$. Suppose that $B_1+ \ldots +B_k=N_k(L)$. Then
\[ B_1+\ldots +B_{k+1}=N_k(L)+N(U_k)=N_k(L)+N_{k+1}(L)\cap U_k=N_{k+1}(L).
\]
\item[(ii)] $B_1+\ldots +B_n=N_n(L)=L$.
\item[(iii)] $[B_i,B_j]=[N(U_i),N(U_j)]\subseteq [N(U_i),U_i] \subseteq N(U_i)=B_i$.
\item[(iv)] We have $U_k=U_{k+1}+N_{k+1}(L)\cap U_k=U_{k+1}+N(U_k)=U_{k+1}+B_{k+1}$. Hence
$U_i=U_{i+1}+B_{i+1}=U_{i+2}+B_{i+2}+B_{i+1}= \dots = B_{i+1}+\dots +B_n$, and $N_i(L)\cap U_i\subseteq \phi(U_i)$ from Lemma \ref{l:decomp}.
\end{itemize}

The converse is clear.
\end{proof}

\section{Extreme Lie Algebras} 
The Lie algebra $L$ is {\em monolithic} if it has a unique minimal ideal $A$, the {\em monolith} of $L$. If $B$ is an ideal of $L$ and $A/B$ is a minimal ideal of $L/B$ we say that $A/B$ is a {\em chief factor} of $L$. The series
\[ \{0\} = A_0 \subset A_1 \subset \ldots \subset A_n=L
\] is called a {\em chief series} if $A_i/A_{i-1}$ is a chief factor of $L$ for each $1\leq i\leq n$.

\begin{lemma}\label{l:chief} Let $L$ be a Lie algebra such that $N(L)/\phi(L)$ is a chief factor of $L$. Then $L$ has at most one complemented minimal ideal, and if $A$ is one such, then $\phi(L/A)=\phi_2(L)/A$.
\end{lemma}
\begin{proof} If $\phi(L)=0$, then $N(L)$ is the monolith and the result follows easily from \cite[Theorems 7.3 and 7.4]{frat}. So assume that $\phi(L) \neq 0$ and let $A$ be a complemented minimal ideal of $L$. Then $A \not \subseteq \phi(L)$, so $A+\phi(L)=N(L)$. Let $M$ be a maximal subalgebra of $L$, and suppose that $\phi_2(L) \not \subseteq M$. Then $N(L) \not \subseteq M$, so $L=M+N(L)= M+A+\phi(L)$, giving $M+A=L$. Thus $A \not \subseteq M$. It follows that every maximal subalgebra of $L/A$ contains $\phi_2(L)/A$. Put $T/A= \phi(L/A)$. Then $\phi_2(L) \subseteq T$.
\par

Let $M/N(L)$ be a maximal subalgebra of $L/N(L)$ and suppose that $(T+N(L))/N(L) \not \subseteq M/N(L)$. Then $M+T+N(L)=M+T=L$. But now $T \not \subseteq M$, so $\phi(L/A) \not \subseteq M$, whence $A \not \subseteq M$. It follows that $L=M+A=M$, a contradiction. Hence $\phi_2(L)/N(L)= \phi(L/N(L)) \supseteq (T+N(L))/N(L)$, which yields that $T= \phi_2(L)$.
\par

Now suppose that $B$ is another minimal ideal of $L$ with $B \neq A$. Then
\[ B \cong (A+B)/A \subseteq N_2(L)/A=N(L/A),
\]
by the above. It follows that $N_2(L) \subseteq C_L(B)$. Suppose that $B \not \subseteq \phi(L)$. Then, as before, we have that $B+ \phi(L)=N(L)$ and hence that $B \cong N(L)/\phi(L)$. But $C_L(N(L)/\phi(L))=N(L)$, by \cite[Theorem 7.4]{frat}, since $N(L/\phi(L))=N(L)/\phi(L)$. This yields that $N(L)=C_L(B)$, a contradiction. Hence $B \subseteq \phi(L)$ and $L$ has, at most, one complemented minimal ideal.
\end{proof}
\bigskip

We call $L$ {\em extreme} if $N_i(L)/\phi_i(L)$ is a chief factor of $L$ for each $i=1,2, \ldots , n(L).$ 

\begin{lemma}\label{l:factor} Every factor algebra of an extreme Lie algebra is extreme.
\end{lemma}
\begin{proof} Let $B$ be an ideal of the extreme Lie algebra $L$. Suppose first that $B \subseteq \phi(L)$. Then  $N_i(L/B)/\phi_i(L/B)$ is a chief factor of $L/B$ for each $i$, by Lemma \ref{l:nat}, and $L/B$ is extreme. So suppose that $B \not \subseteq \phi(L)$. Then $\phi(L/B)=\phi_2(L)/B$, by Lemma \ref{l:chief}, and the result again follows.
\end{proof}
\bigskip

We say that the chief factor $A/B$ is {\em complemented} if there is a maximal subalgebra $M$ of $L$ such that $L=A+M$ and $A\cap M=B$. We define $c(L)$ to be the number of complemented chief factors in a chief series for $L$. This is independent of the particular chief series chosen, by \cite[Theorem 2.3]{chief}. 
\par

Let $x \in L$ and let ad\,$x$ be the corresponding inner derivation of $L$. If $F$ has characteristic zero suppose that (ad\,$x)^n = 0$ for some $n$; if $F$ has characteristic $p$ suppose that $x \in I$ where $I$ is a nilpotent ideal of $L$ of class less than $p$. Put
\[
\hbox{exp(ad\,}x) = \sum_{r=0}^{\infty} \frac{1}{r!}(\hbox{ad\,}x)^r.
\]
Then exp(ad\,$x)$ is an automorphism of $L$. We call the group ${\mathcal I}(L)$ generated by all such automorphisms the group of {\em inner automorphisms} of $L$. Two subsets $U, V$  are {\em conjugate in $L$} if  $U = \alpha(V)$ for some $\alpha \in {\mathcal I}(L)$.
\par

Then we have the following characterisation of extreme Lie algebras.

\begin{theor}\label{t:extreme}  Let $L$ be a solvable Lie algebra. Then the following statements are equivalent:
\begin{itemize}
\item[(i)] $L$ is extreme;
\item[(ii)] $n(L)=m(L)$, the number of conjugacy classes of maximal subalgebras of $L$;
\item[(iii)] $n(L)=c(L)$; and
\item[(iv)] if $B$ is an ideal of $L$, then $L/B$ has at most one complemented minimal ideal.
\end{itemize}
\end{theor}
\begin{proof} \begin{itemize} \vspace{-.2cm} 
\item[(i) $\Rightarrow$ (ii)]: Let $L$ be extreme and consider the series
\[ 0 \subset \phi_1(L) \subset N_1(L) \subset \ldots \subset \phi_i(L) \subset N_i(L) \subset \ldots.
\]
There is a unique conjugacy class of maximal subalgebras of $L$ complementing the chief factor $N_i(L)/\phi_i(L)$ for each $i=1,2, \ldots, n(L)$, by \cite{barnes}. But each maximal subalgebra of $L$ must complement one of the complemented chief factors in the above series, and must, therefore, belong to one of these $n(L)$ conjugacy classes. Hence $n(L)=m(L)$.
\item[(ii) $\Rightarrow$ (iii)]: We use induction on the dimension of $L$. Suppose that $L$ is a Lie algebra satisfying $n(L)=m(L)$ and assume that the implication holds for Lie algebras of smaller dimension than that of $L$. If $\phi(L) \neq 0$, we have $n(L/\phi(L))=n(L)=m(L)=m(L/\phi(L)$, and so, by induction, $n(L)=n(L/\phi(L))=m(L/\phi(L))=c(L/\phi(L))=c(L)$.
\par

So suppose that $\phi(L)=0$. Then $N(L)=Asoc(L)$ and each of the $r$ (say) minimal ideals in $Asoc(L)$ is complemented, by \cite[Theorem 7.4 and Lemma 7.2]{frat}. It follows that $n(L)=m(L) \geq m(L/N(L))+r \geq n(L)-1+r$. Hence $r=1$ and, by induction, $c(L)=1+c(L/N(L))=1+n(L/N(L))=n(L)$.
\item[(iii) $\Rightarrow$ (i)]: This follows from the fact that there is at least one complemented chief factor $A/B$ satisfying $\phi_i(L) \leq B < A \leq N_i(L)$ for each $i=1,2, \ldots, n(L)$.
\item[(i) $\Leftrightarrow$ (iv)]: If $L$ is extreme, then so is $L/B$, by Lemma \ref{l:factor}. Hence $L/B$ has at most one complemented minimal ideal, by Lemma \ref{l:chief}.
\par

Conversely, suppose that $L$ satisfies (iv) and consider $L/\phi_i(L)$. Since $N_i(L)/\phi_i(L)$ is the direct sum of complemented minimal ideals of $L/\phi_i(L)$, as above, it follows that $N_i(L)/\phi_i(L)$  is a chief factor of $L/\phi_i(L)$. Hence $L$ is extreme.
\end{itemize}
\end{proof}
\bigskip

\begin{lemma}\label{l:extreme} Let $L$ be an extreme Lie algebra.
\begin{itemize}
\item[(i)] If $L$ is nilpotent, then $\dim L=1$.
\item[(ii)] If $L \in {\mathcal N}(n)$ then $N_{n-1}(L)=\Gamma_1(L)=L^k$ has codimension one in $L$, for all $k\geq 2$.
\end{itemize}
\end{lemma}
\begin{proof}\begin{itemize}
\item[(i)] Since $L$ is nilpotent, $\phi(L)=L^2$, and so $\dim L/L^2=1$. The result follows.
\item[(ii)] This follows from Lemma \ref{l:factor} and (i).
\end{itemize}
\end{proof}

\begin{theor}\label{t:extremedecomp}
The Lie algebra $L \in {\mathcal N}(n)$ is extreme if and only if $L$ has the decomposition given in Theorem \ref{t:decomp}, $\dim B_n=1$ and $N(U_k)/\phi(U_k)$ is a chief factor of $U_k$ for each $k=0, \ldots,n-1$.
\end{theor}
\begin{proof} We have that
\begin{align}
\frac{N_{k+1}(L)}{\phi_{k+1}(L)} & =\frac{N_{k+1}(L)\cap U_k+N_k(L)}{\phi_{k+1}(L)\cap U_k+N_k(L)}\cong \frac{(N_{k+1}(L)\cap U_k+N_k(L))/N_k(L)}{(\phi_{k+1}(L)\cap U_k+N_k(L))/N_k(L)} \nonumber \\
 & \cong \frac{N_{k+1}(L)\cap U_k/N_k(L)\cap U_k}{\phi_{k+1}(L)\cap U_k/N_k(L)\cap U_k} \cong \frac{N_{k+1}(L)\cap U_k}{\phi_{k+1}(L)\cap U_k}= \frac{N(U_k)}{\phi(U_k)}. \nonumber
\end{align}
Also $\dim B_n=1$ by Lemma \ref{l:extreme} (ii).
\end{proof} 
\bigskip

If $S$ is a subalgebra of $L$, we will denote by $\overline{S}$ the image of $S$ under the canonical homomorphism from $L$ onto $L/\phi(L)$. We have the following characterisation of those Lie algebras $L \in {\mathcal N({\leq 2})}$ that are extreme, which includes all extreme Lie algebras over a field of characteristic zero.

\begin{coro}\label{c:extreme2} Let $L \in {\mathcal N({\leq 2})}$. Then $L$ is extreme if and only if one of the following holds.
\begin{itemize}
\item[(i)] $\dim L=1$; or
\item[(ii)] $\overline{L}=\overline{A} \dot{+} \overline{U}$ where $\overline{A}=\overline{N(L)}$ is the monolith of $\overline{L}$ and $\overline{U}$  is a one-dimensional subalgebra of $\overline{L}$ which acts irreducibly on $\overline{A}$.
\end{itemize}
\end{coro}
\begin{proof} This is just the cases $n=1,2$ in Theorem \ref{t:extremedecomp}.
\end{proof}

\begin{coro}\label{c:supsolv} Let $L$ be supersolvable. Then $L$ is extreme if and only if one of the following holds.
\begin{itemize}
\item[(i)] $\dim L=1$; or
\item[(ii)] $L/\phi(L)$ is the two-dimensional non-abelian Lie algebra.
\end{itemize}
\end{coro}
\begin{proof}  Let $L$ be supersolvable and extreme. Then $\dim \overline{N(L)}=1$ and so $\dim \overline{L}/C_{\overline{L}}(\overline{N(L)})=1$. But $\overline{N(L)}=N(\overline{L})$, so $C_{\overline{L}}(\overline{N(L)}) \subseteq \overline{N(L)}$. It follows that $\dim (L/N(L)) \leq 1$ and $L \in {\mathcal N({\leq 2})}$. But now either $\dim L=1$ or $\overline{L}=\overline{A} \dot{+} \overline{U}$ where $\dim \overline{A}= \dim \overline{U}=1$, by Theorem \ref{t:extremedecomp}. In the latter case $\dim \overline{L}=2$ and $\overline{L}$ cannot be abelian.
\par

Conversely, if $\dim (L/\phi(L)) \leq 2$ then $L/\phi(L)$ is supersolvable, and so $L$ is supersolvable, by \cite[Theorem 6]{barnes}. Clearly $L$ is also extreme in each of cases (i) and (ii). 
\end{proof}

\begin{ex}\label{e:supsolv} It is easy to check that every three-dimensional Lie algebra as described in Corollary \ref{c:supsolv} has a basis $x, y, z$ with non-zero products $[x,y]=y+z$, $[x,z]=\alpha z$ for some $0 \neq \alpha \in F$. Moreover, no two of these with different values of $\alpha$ are isomorphic. 
\end{ex}

\section{Minimal non-${\mathcal N}$ algebras}
If $\mathcal X$ is a class of Lie algebras, we say that $L$ is {\em minimal non-$\mathcal X$} if every proper subalgebra of $L$, but not $L$ itself, belongs to $\mathcal X$. We say that  $L$ is {\em minimal non-$\mathcal N$} if it is minimal non-$\mathcal N(\leq k)$ for some $k$; in other words, if its nilpotent length is greater than that of any of its proper subalgebras. Over a field of characteristic zero a Lie algebra can only be minimal non-${\mathcal N(1)}$ and these are described in \cite{nilp}. 

\begin{lemma}\label{l:cod} Let $L$ be minimal non-${\mathcal N(\leq k-1)}$ and let $M$ be a maximal subalgebra of $L$. Then 
\begin{itemize}
\item[(i)] $L^2=N_{k-1}(L)$ has codimension one in $L$; 
\item[(ii)] if $N_i(M) \not \subseteq N_{k-1}(L)$ then $N_{k-1}(L) \cap N_i(M)$ has codimension one in $N_i(M)$ for $i=1, \ldots, k-1$; and
\item[(iii)] if $N(L) \subseteq M$ then $N(M) \subseteq N_{k-1}(L)$.
\end {itemize}
\end{lemma}
\begin{proof} \begin{itemize}
\item[(i)] Let $M$ be a maximal subalgebra of $L$ containing $N_{k-1}(L)$.  Since $L/N_{k-1}(L)$ is nilpotent, $M$ is an ideal of $L$ and has codimension one in $L$. But $N_{k-1}(M)_L=N_{k-1}(L)$, by Lemma \ref{l:max} (ii), so, if $N_{k-1}(L) \neq M$, then $M$ has nilpotent length $k$, contradicting the fact that it is minimal non-${\mathcal N(k)}$. Hence $M=N_{k-1}(L)$. 
\par

Now let $M$ be any maximal subalgebra containing $L^2$, so $M$ is an ideal of codimension one in $L$. We have $M \cap N_i(L)=N_i(M)_L$ for each $i=1, \ldots k-1$ by Lemma \ref{l:max} (i) and (ii). It follows that $M \cap N_{k-1}(L) = N_{k-1}(M)_L=M$, so $M \subseteq N_{k-1}(L)$, whence $M=N_{k-1}(L)=L^2$.
\item[(ii)]Suppose $N_i(M) \not \subseteq N_{k-1}(L)$. Then $L=N_{k-1}(L)+N_i(M)$, so
\[ \frac{L}{N_{k-1}(L)} \cong \frac{N_i(M)}{N_{k-1}(L) \cap N_i(M)},
\]
whence the result. 
\item[(iii)] Let $N(L) \subseteq M$. Then $N(M)$ acts nilpotently on $L$, by Lemma \ref{l:max} (v). If $N(M) \not \subseteq N_{k-1}(L)$ then $L=N_{k-1}(L)+N(M)$ and $L/N_{k-2}(L)$ is nilpotent, a contradiction.
\end{itemize}
\end{proof}

In group theory, every minimal non-${\mathcal N(k)}$ group is extreme, and so a natural question is whether this holds for Lie algebras. We show next that this is `usually' the case for Lie algebras. We call a class ${\mathcal H}$ of Lie algebras a {\em semi-homomorph} if, for all $L \in {\mathcal H}$,
\begin{itemize}
\item[(i)] $L/N(L) \in {\mathcal H}$; and
\item[(ii)] if $A$ is an ideal of $L$ and $A\subseteq \phi(L)$, then $L/A \in {\mathcal H}$.
\end{itemize}

\begin{lemma}\label{l:crit} Let ${\mathcal H}$ be the class of Lie algebras $L$ in which all maximal subalgebras of $L$ have strongly nilregular cores. Then ${\mathcal H}$ is a semi-homomorph.
\end{lemma}
\begin{proof} 
\begin{itemize} \item[(i)] Let $M/N(L)$ be a maximal subalgebra of $L/N(L)$. Clearly we have that $(M/N(L))_{L/N(L)}=M_L/N(L)$ and $M_L/N(L)$ has compatibility index one less than that of $M_L$, $r-1$ say. Then 
\[ \frac{N_i(M_L/N(L))}{N_{i-1}(M_L/N(L))}\cong \frac{N_{i+1}(M_L)}{N_i(M_L)},
\] which has nilpotency class $<p-1$ for $i=1, \ldots, r-1$, since $M_L$ is strongly nilregular. Hence $M/N(L)$ has a strongly nilregular core.
\item[(ii)] Let $M/A$ be a maximal subalgebra of $L/A$. Then $(M/A)_{L/A}=M_L/A$. Also, $N_i(L/A)=N_i(L)/A$, by Lemma \ref{l:nat}, so $M_L/A$ has the same compatibility index as $M_L$, $r$ say. Also $N_i(M)=N_i(M_L)=N_i(L)$ for $i=1, \ldots, r$, by Proposition \ref{p:nmax}.
\par

We claim that $N_i(M_L/A)=N_i(M_L)/A$ for each $i=1, \ldots, n(M_L)$.  Put $K_i/A=N_i(M_L/A)$, so $K_{i+1}/K_i=N(M_L/K_i)$. Then $K_1$ is a subideal of $L$, $A\subseteq K_1\cap \phi(L)$ and $K_1/A$ is nilpotent, so $K_1$ is nilpotent, by \cite[Theorem 6.1]{frat}. It follows that $K_1\subseteq N(M_L)$ and $N(M_L/A)=N(M_L)/A$. Hence 
\[ K_1 \subseteq N(M_L) \subseteq K_2 \subseteq \ldots \subseteq K_i \subseteq N_i(M_L) \subseteq K_{i+1} \subseteq N_{i+1}(L) \subseteq \ldots
\] and $K_{i+1}/K_i\subseteq N_{i+1}(L)/K_i=N_{i+1}(M_L)/K_i$ for $i=1, \ldots, r-1$. It follows that $N_{i+1}(M_L/A)=K_{i+1}/A\subseteq N_{i+1}(L)/A$. The reverse inclusion is clear and the claim is established.
\par

Now we have that
\[ \frac{N_i(M_L/N(L))}{N_{i-1}(M_L/N(L))}\cong \frac{N_i(M_L)}{N_{i-1}(M_L)},
\] which has nilpotency class $<p-1$ for $i=1, \ldots, r$, since $M_L$ is strongly nilregular. Hence $M/N(L)$ has a strongly nilregular core.
\end{itemize}
\end{proof}

A solvable primitive algebra has a unique minimal, self-centralising, ideal $A$ such that $L=A\dot{+} U$ (see \cite{chief}). We shall say that a class of Lie algebras ${\mathcal H}$ has the {\em primitive quotient property} if, for every primitive algebra $L$ in ${\mathcal H}$ with minimal ideal $A$, $L/A$ is minimal non-${\mathcal N}$. 

\begin{theor}\label{t:crit} Let ${\mathcal H}$ be a semi-homomorph with the primitive quotient property, and let $L \in {\mathcal H}$ be a Lie algebra which is minimal non-${\mathcal N(\leq n)}$. Then 
\begin{itemize}
\item[(i)] $L$ is extreme; and
\item[(ii)] $L/N(L)$ is minimal non-${\mathcal N(\leq n-1)}$.
\end{itemize}
\end{theor}
\begin{proof}  \begin{itemize} \item[(i)] We use induction on $\dim L$. Suppose first that $\phi(L) \neq 0$. Let $A$ be a minimal ideal of $L$ contained in $\phi(L)$. Then $n(L/A)=n(L)$, so $L/A$ is minimal non-$\mathcal N$. By the inductive hypothesis, $L/A$, and hence $L$ is extreme. So suppose that $\phi(L)=0$. If there are at least two minimal ideals then there is at least one, $A$ say, such that $n(L/A)=n(L)$, by Corollary \ref{c:satfor}. But then $L=A \dot{+} M$ for some maximal subalgebra $M$ of $L$, and $n(M)=n(L/A)=n(L)$, a contradiction. 
\par

Thus there is a unique minimal ideal $A=N(L)$, $L$ is primitive and $n(L/A)=n(L)-1$. Since ${\mathcal H}$ has the primitive quotient property, $L/A$ is minimal non-${\mathcal N}$. Moreover, since ${\mathcal H}$ is a semi-homomorph, $L/A$, and thus $L$, is extreme, by induction. 
\item[(ii)]  Consider the series
\[ 0 \subset \phi_1(L) \subset N_1(L) \subset \ldots \subset \phi_i(L) \subset N_i(L) \subset \ldots,
\]
and let $M$ be a maximal subalgebra of $L$ containing $N_1(L)$. Then $M$ must complement one of the complemented chief factors $N_k(L)/\phi_k(L)$ for some $2 \leq k \leq n$ in the above series. But then $L=N_k(L)+M$, $M \cap N_k(L)=\phi_k(L)$ and $N_i(M)=N_i(L)$ for $i=1, \ldots, k-1$, by Lemma \ref{l:max}. Thus $n(M)-k+1=n(M/N_{k-1}(L))=n(L/N_k(L))=n(L)-k$, whence $n(M/N_1(L))=n(M)-1=n(L)-2=n(L/N_1(L))-1$ and $L/N_1(L)$ is minimal non-${\mathcal N(n-1)}$.
\end{itemize}
\end{proof}

\begin{coro}\label{c:crit}  Let $L$ be a Lie algebra in which all maximal subalgebras have strongly nilregular cores and which is minimal non-${\mathcal N(\leq n)}$. Then $L$ is extreme and $L/N(L)$ is minimal non-${\mathcal N(\leq n-1)}$.
\end{coro}
\begin{proof} Let ${\mathcal H}$ be the class of Lie algebras whose maximal subalgebras have strongly nilregular cores. Then ${\mathcal H}$ is a semi-homomorph, by Lemma \ref{l:crit}.Let $L$ be a primitive algebra in ${\mathcal H}$ with minimal ideal $A=N(L)$, and let $M$ be a maximal subalgebra containing $A$. Then $N(M)=A$ by Proposition \ref{p:nmax}, so $L/A$ is minimal non-${\mathcal N(\leq n-1)}$ and ${\mathcal H}$ satisfies the quotient primitive property. The result now follows from Theorem \ref{t:crit}.
\end{proof}
\bigskip

A Lie algebra $L$ is called an $A$-algebra if all of its nilpotent subalgebras are abelian. These arise in the study of constant Yang–Mills potentials and in relation to the problem of describing residually finite varieties. The structure of solvable Lie $A$-algebras was studied in some detail in \cite{Aalg}. In the case of an $A$-algebra the lower nilpotent series and the derived series coincide (\cite[Lemma 2.3]{Aalg}), and so the terms ``derived length" and ``nilpotent length" are identical.

\begin{coro}\label{c:Acrit} If $L$ is an $A$-algebra which is minimal non-${\mathcal N(\leq n)}$, then $L$ is extreme and $L/N(L)$ is minimal non-${\mathcal N(\leq n-1)}$.
\end{coro}
\begin{proof} Let ${\mathcal H}$ be the class of $A$-algebras. Then ${\mathcal H}$ is a semi-homomorph, by \cite[Lemma 2.1 (iii)]{Aalg}. Let $L \in {\mathcal H}$ be primitive with minimal ideal $A=N(L)$ and let $M$ be a maximal subalgebra containing $A$. Then $N(M)$ is abelian and so $[N(M),A]=0$, giving $N(M) \subseteq C_L(A)=A$. It follows that $N(M)=A$ and so ${\mathcal H}$ has the primitive quotient property.
\end{proof}

\begin{coro}\label{c:2crit} Let $L$ be minimal non-${\mathcal N(\leq 2)}$ and have solvability index $\leq 3$. Then $L$ is extreme and $L/N(L)$ is minimal non-${\mathcal N(\leq 1)}$.
\end{coro}
\begin{proof}  Let ${\mathcal H}$ be the class of Lie algebras of solvability index $\leq 3$. Then ${\mathcal H}$ is clearly a semi-homomorph. Let $L \in {\mathcal H}$ be primitive with minimal ideal $A=N(L)$ and let $M$ be a maximal subalgebra containing $A$. If $L$ has solvability index $\leq 2$ it is clear that $L/A$ is minimal non-${\mathcal N}$, so assume that $L$ has index $3$.
\par

We have $L^{(1)}=N_2(L)$, by Lemma \ref{l:cod} (i) and $L^{(2)} \subseteq N(L)=A$, so $N_2(L)/A$ is abelian. Now $N(M) \subseteq N_2(L)$ by Lemma \ref{l:cod} (iii), and so $$[N_2(L),N(M)] \subseteq N_2(L)^2 \subseteq A \subseteq M.$$ Then either $M \neq N_2(L)$ or $L=N_2(L)+M$, in which case $N(M)$ is an ideal of $L$ and $N(M)=A$. In either case $M/A$ is nilpotent and so ${\mathcal H}$ has the primitive quotient property.
\end{proof}

\begin{ex} Note that there are extreme Lie algebras which are not minimal non-${\mathcal N}$. For example, let $L$ be the Lie algebra over any field $F$ with basis $x, y, z$ with non-zero products $[x,y]=y+z$, $[x,z]=z$. Then $\phi(L)=Fz$ and $N(L)=Fy+Fz$, so $L$ is extreme. However, if $M=Fx+Fz$ then $n(M)=2=n(L)$.
\end{ex}
\medskip

Next we seek to characterise the algebras considered in Corollaries \ref{c:Acrit} and \ref{c:2crit}. We can characterise the $A$-algebras that are also minimal non-${\mathcal N}$ as follows.

\begin{theor}\label{Aalg} Let $L$ be a Lie $A$-algebra of derived length $n+1$ over a field $F$. Then $L$ is minimal non-${\mathcal N}$ if and only if the following hold.
\begin{itemize}
\item[(i)] $L= A_n \dot{+} A_{n-1} \ldots \dot{+} A_1 \dot{+} Fx$ where $A_i$ is an abelian subalgebra of $L$ for each $1 \leq i \leq n$;
\item[(ii)] $L^{(i)}= A_n \dot{+} A_{n-1} \ldots \dot{+} A_i$ for each $1 \leq i \leq n$;
\item[(iii)] $[A_i,A_j] \subseteq A_j$ for $j>i$; 
\item[(iv)] $A_i$ is an irreducible $L/L^{(i+1)}$-module for each $0 \leq i \leq n$; and
\item[(v)] $N_{n-i+1}(L)=L^{(i)}$ for each $0 \leq i \leq n$.
\end{itemize}
\end{theor}
\begin{proof} Suppose first that $L$ is minimal non-${\mathcal N}$. Since $L$ is an $A$-algebra of derived length $n+1$, $L= A_n \dot{+} A_{n-1} \ldots \dot{+} A_1 \dot{+} A_0$ and  $L^{(i)}= A_n \dot{+} A_{n-1} \ldots \dot{+} A_i$, where $A_i$ is an abelian subalgebra of $L$ for each $0 \leq i \leq n$, by \cite[Corollary 3.2]{Aalg}. But $\dim A_0=1$, by Lemma \ref{l:cod} (i), so $A_0 =Fx$ for some $x \in L$. This gives (i) and (ii). The decomposition in (i) follows from the splitting of a Lie $A$-algebra over each term in its derived series (\cite[Theorem 3.1]{Aalg}), so $L=A_n \dot{+} B_n$, where $A_n=L^{(n)}$, $B_n=A_{n-1} \dot{+} B_{n-1}$, where $A_{n-1}=B_n^{(n-1)}$, and so on. But now $[A_i,A_j] \subseteq L^{(j)} \cap B_{j+1} \subseteq A_j$ if $j>i$, giving (iii).
\par

A straightforward induction argument shows that $A_i \subseteq N_{n-i-1}$ for $0 \leq i \leq n$. We now establish (iv) and (v) by induction on $n$. Then (iv) clearly holds for $i=0$, and (v) holds for $i=0$ by Lemma \ref{l:cod} (i), so suppose that they hold for all $i \geq k$ ($k \geq 0$). Then $N_{n-k+1}(L)=L^{(k)}=A_n \dot{+} \ldots \dot{+} A_k$ and $L^{(k+1)}=A_n \dot{+} \ldots \dot{+} A_{k+1} \subseteq N_{n-k}(L)$. It follows that $N_{n-k}(L)=L^{(k+1)}$ by the irreducibility of $A_k$ and the fact that $N_{n-k}(L) \neq N_{n-k+1}(L)$ (since $L$ has nilpotent length $n+1$).
\par

Also $\phi_{n-k}(L) \subseteq L^{(k+1)}=N_{n-k}(L)$, and $N_{n-k}(L)/\phi_{n-k}(L)$ is irreducible, since $L$ is extreme, by Theorem \ref{t:crit}. Hence $M=\phi_{n-k}(L) \dot{+} A_{k} \dot{+} \ldots \dot{+} A_1 \dot{+} Fx$ is a maximal subalgebra of $$A_{k+1} \dot{+} A_k \dot{+} \ldots \dot{+} A_1 \dot{+} Fx \cong L/L^{(k+2)} = L/N_{n-k-1}(L).$$
But $L/N_{n-k-1}(L)$ is minimal non-${\mathcal N}(k+1)$, by Theorem \ref{t:crit} (ii), so $M \in {\mathcal N}(k+1)$. It follows that $N(M)= \phi_{n-k}(L) \dot{+} A_k$, so $[\phi_{n-k}(L),A_k]=0$. But $A_k$ is a Cartan subalgebra of $A_{k+1} \dot{+} A_k$, by \cite[Theorem 3.1]{Aalg}, so $\phi_{n-k}(L)=0$ and $A_k$ is an irreducible $L/L^{(k+1)}$-module. This establishes (iv) and (v).
\par

Conversely, suppose that (i)-(v) hold and let $M$ be a maximal subalgebra of $L$. Clearly $L \in {\mathcal N}(n+1)$. Let $i$ be the smallest integer such that $L^{(i)} \not \subseteq M$. Then $L^{(i)}/L^{(i+1)}$ is a minimal ideal of $L/L^{(i+1)}$, so $M/L^{(i+1)}$ complements $L^{(i)}/L^{(i+1)}$ in $L/L^{(i+1)}$. Hence $M^{(i)} \subseteq M \cap L^{(i)} \subseteq L^{(i+1)}$. But $n(L^{(i+1)})<n-i+1$ by (v). Hence $n(M)<n+1$ and $L$ is minimal non-${\mathcal N}$.
\end{proof}
\bigskip

 Recall the following result from \cite{aN-by-A}. 

\begin{theor}(\cite[Theorem 4]{aN-by-A})\label{t:aN-by-A} Let $L$ be solvable and $\phi$-free. Then $L$ is minimal non-(nilpotent-by-abelian) if and only if $F$ has characteristic $p>0$ and $L=A\dot{+} B$, where $A$ is the unique minimal ideal of $L$, $\dim A \geq 2$, $A^2=0$, and either $B=M\dot{+} Fx$, where $M$ is a minimal ideal of $B$ such that $M^2=0$ (type I), or $B$ is the three-dimnsional Heisenberg algebra (type II). Moreover, if $p \geq 3$ then $\dim A$ is divisible by $p$.
\end{theor}
\bigskip

Then we have the following characterisation of the algebras of solvability index $3$ which are minimal non-${\mathcal N}(2)$.

\begin{theor}\label{t:solv3} Let $L$ be a solvable Lie algebra of solvability index $3$. Then $L$ is minimal non-${\mathcal N}(2)$ if and only if it is minimal non-(nilpotent-by-abelian) of type I.
\end{theor}
\begin{proof} Let $L$ be minimal non-${\mathcal N}(2)$, and denote the image of a subalgebra $S$ of $L$ under the natural homomorphism onto $L/\phi(L)$ by $\overline{L}$. Then  $\overline{L}=\overline{N(L)}\dot{+} \overline{U}$ where $U$ is a subalgebra of $L$, by \cite[Theorems 7.3 and 7.4]{frat}. Moreover, $\overline{U}= \overline{A}\dot{+}F\overline{x}$ where $\overline{A}$ is abelian, $N_2(L)=N(L)+A$ and $L/N(L)$ is minimal non-nilpotent, by Corollary \ref{c:2crit}. It follows from \cite[Theorem 2.1]{nilp} that $N_2(L)/N(L)$ is irreducible. Now $U$ is a maximal subalgebra of $L$, so $U \in {\mathcal N}(2)$, which yields that $N(U)=A$ and $U$ is nilpotent-by-abelian.

Let $M$ be any maximal subalgebra of $L$. Then either $M\cong U$ or $N(L)\subseteq M$. Suppose that $N(L)\subseteq M$. Then $M=N_2(L)$ or $L=N(L)+Fx$. In either case $M$ is nilpotent-by-abelian. Hence $L$  is minimal non-(nilpotent-by-abelian). It is of type I, since otherwise $L \in {\mathcal N}(2)$.
\par

Conversely, suppose that $L$ is minimal non-(nilpotent-by-abelian) of type I. Then $L$ has solvability index $3$ and the maximal subalgebras of $L$ are nilpotent-by-abelian, as in the paragraph above. Clearly $L$ itself is not nilpotent-by-abelian.
\end{proof}
\bigskip

Lie algebras as described in Theorem \ref{t:solv3} do exist over every field of characteristic $p>0$, as is shown in \cite{aN-by-A}; over an algebraically closed field they are minimal non-supersolvable (\cite[Theorem 5]{aN-by-A}). Finally we show that a homomorphic image of a minimal non-${\mathcal N}$ Lie algebra is minimal non-${\mathcal N}$ if it has a complemented minimal ideal.
 
\begin{theor}\label{t:chief} If $L$ is a minimal non-${\mathcal N}$ Lie algebra in which all maximal subalgebras have nilregular cores, and $A/B$ is a complemented chief factor of $L$, then $L/B$ is minimal non-${\mathcal N}$.
\end{theor}
\begin{proof} Suppose there is a chief series of $L$ through $A$ and $B$ in which $A/B$ is the $k$th complemented factor, where $1 \leq k \leq c(L)$. If $k=1$, then $B \subseteq \phi(L)$ and so $n(L/B)=n(L)$, in which case $L/B$ is minimal non-${\mathcal N}$ because $L$ is.
\par

So let $k >1$ and assume that the theorem holds for the $(k-1)$th complemented factor $C/D$. Without loss of generality we may assume that $D=0$. Then $A/B$ is the second complemented factor in some chief series of $L$. It follows from Theorem \ref{t:crit} that $L$ and $L/N(L)$ are extreme, and so every chief series of $L$ has only one complemented chief factor below $N(L)$, by Theorem \ref{t:extreme}. If $N(L) \cap A \not \subseteq B$ then $N(L) \cap A/N(L) \cap B$ is a complemented chief factor of $L$ and $L$ would have a chief series with two complemented chief factors below $N(L)$, a contradiction. Hence $N(L) \cap A \subseteq B$.
\par

Let $n=n(L)$. Then $L/A$ and $L/B$ are both extreme, by Lemma \ref{l:factor} and so
\[ n(L/B)=c(L/B)=n-1 \hspace{2mm} \hbox{ and } \hspace{2mm} n(L/A)=c(L/A)=n-2,
\]
by Theorem \ref{t:extreme}. Let $M/B$ be a maximal subalgebra of $L/B$. We have $B \not \subseteq \phi(L)$ and $N(L)/\phi(L)$ is a chief factor of $L$, so $M \supseteq \phi(L)+B \supseteq N(L)$. But $L/N(L)$ is minimal non-${\mathcal N}$, by Theorem \ref{t:crit}, so $n(M/N(L)) \leq n-2$. It follows that $N_{n-2}(M) \subseteq N(L) \cap A \subseteq B$, whence $n(M/B) \leq n-2 < n(L/B)$ and $L/B$ is minimal non-${\mathcal N}$, as claimed.
\end{proof}

\end{document}